\documentclass[12pt,a4paper]{article}
\usepackage{amsmath,amsfonts,amssymb,amsthm}

\usepackage{graphicx}
\usepackage[colorlinks=true,dvipdfm]{hyperref}

\newtheorem{remark}{Remark}
\newtheorem{corollary}{Corollary}
\newtheorem{thm}{Theorem}

\newtheorem{lem}{Lemma}
\newtheorem{prop}{Proposition}

\theoremstyle{definition}
\newtheorem{defin}{Definition}

\theoremstyle{remark}

\begin{document}

\title{Farey neighbors and hyperbolic Lorenz knots}

\author{Paulo Gomes\thanks{\'Area Departamental de Matem\'atica, Instituto Superior de Engenharia de Lisboa, e-mail: pgomes@adm.isel.pt}, Nuno Franco\thanks{CIMA-UE and Departamento de Matem\'atica, Universidade de \'Evora, e-mail: nmf@uevora.pt} and Lu\'is Silva\thanks{CIMA-UE and \'Area Departamental de Matem\'atica, Instituto Superior de Engenharia de Lisboa, e-mail: lfs@adm.isel.pt}}

\maketitle

\begin{abstract}
 Based on symbolic dynamics of Lorenz maps, we prove that, provided one conjecture due to Morton is true, then Lorenz knots associated to orbits of points in the renormalization intervals of Lorenz maps with reducible kneading invariant of type $(X,Y)*S$, where the sequences $X$ and $Y$ are Farey neighbors verifying some conditions, are hyperbolic.
\end{abstract}

\section{Introduction}
\label{sec:intro}

\subsection*{Lorenz knots}
\label{sec:lorknots}

\par

\emph{Lorenz knots} are the closed (periodic) orbits in the Lorenz system \cite{Lorenz63}

\begin{align}
  \label{eq:lorsys}
  x' &= -10 x +10 y \nonumber \\
  y' &= 28 x -y -xz \\
  z' &= -\frac{8}{3} z +xy \nonumber 
\end{align}
while \emph{Lorenz links} are finite collections of (possibly linked) Lorenz knots.

The systematic study of Lorenz knots and links was made possible by the introduction of the \emph{Lorenz template} or knot-holder by Williams in \cite{Williams77} and \cite{Williams79}. It is a branched 2-manifold equipped with an expanding semi-flow, represented in Fig. \ref{fig:lortemp}. It was first conjectured by Guckenheimer and Williams and later proved through the work of Tucker \cite{Tucker02} that every knot and link in the Lorenz system can be projected into the Lorenz template. Birman and Williams made use of this result to investigate Lorenz knots and links \cite{Birman83}. For a review on Lorenz knots and links, see also \cite{Birman11}.

A $T(p,q)$ torus knot is (isotopic to) a curve on the surface of an unknotted torus $T^2$ that intersects a meridian $p$ times and a longitude $q$ times. Birman and Williams \cite{Birman83} proved that every torus knot is a Lorenz knot.

A satellite knot is defined as follows: take a nontrivial knot $C$ (companion) and nontrivial knot $P$ (pattern) contained in a solid unknotted torus $T$ and not contained in a $3-ball$ in $T$. A satellite knot is the image of $P$ under an homeomorfism that takes the core of $T$ onto $C$.

A knot is  hyperbolic if its complement in $S^3$ is a hyperbolic $3-manifold$. Thurston \cite{Thurston82} proved that a knot is hyperbolic \emph{iff} it is neither a satellite knot nor a torus knot. One of the goals in the study of Lorenz knots has been their classification into \emph{hyperbolic} and \emph{non-hyperbolic}, possibly further distinguishing torus knots from satellites.  Birman and Kofman \cite{Birman09} listed hyperbolic Lorenz knots taken from a list of the simplest hyperbolic knots. In a previous article, \cite{Gomes14}, we generated and tested for hyperbolicity, using the program \emph{SnapPy}, families of Lorenz knots that are a generalization of some of those that appear in this list, which led us to conjecture that the families tested are hyperbolic \cite{Gomes13}.

 Morton has conjectured \cite{Elrifai88},\cite{Dehornoy11} that all Lorenz satellite knots are cablings (satellites where the pattern is a torus knot) on Lorenz knots.

In \cite{PhysicaD}, based in the work of El-Rifai, \cite{Elrifai88}, we derived an algorithm to obtain Lorenz satellite braids, together with the corresponding words from symbolic dynamics.

The first-return map induced by the semi-flow on the \emph{branch line} (the horizontal line in Fig. \ref{fig:lortemp}) is called a \emph{Lorenz map}. If the branch line is mapped onto $[-1,1]$, then the Lorenz map $f$ becomes a one-dimensional map from $[-1,1] \setminus \{0\}$ onto $[-1,1]$, with one discontinuity at $0$ and stricly increasing in each of the subintervals $[-1,0[$ and $]0,1]$.

Periodic orbits in the flows correspond to periodic orbits on the Lorenz maps, so symbolic dynamics of the Lorenz maps provide a codification of the Lorenz knots. In \cite{DCDS}, using this codification, it was introduced an operation over Lorenz knots that is directly related with renormalization of Lorenz maps.

In this work we prove that some families of knots, generated from torus knots through this operation, are hyperbolic.

\begin{figure}
  \centering
  \includegraphics[scale=.28]{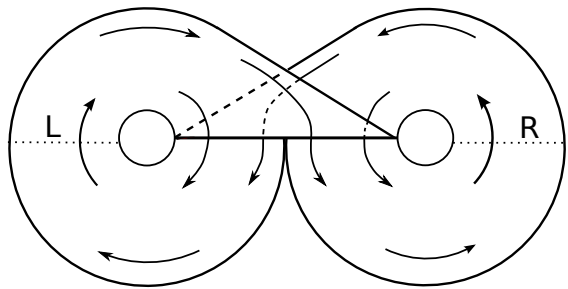}
  \caption{The Lorenz template}
  \label{fig:lortemp}
\end{figure}

\subsection*{Lorenz braids}
\label{sec:lorbraids}

If the Lorenz template is cut open along the dotted lines in Fig. \ref{fig:lortemp}, then each knot and link on the template can be obtained as the closure of an open braid on the cut-open template, which will be called the \emph{Lorenz braid} associated to the knot or link (\cite{Birman83}). These \emph{Lorenz braids} are simple positive braids (our definition of positive crossing follows Birman and is therefore opposed to an usual convention in knot theory). Each Lorenz braid is composed of $n=p+q$ strings, where the set of $p$ left or $L$ strings cross over at least one (possibly all) of the $q$ right strings, with no crossings between strings in each subset. These sets can be subdivided into subsets $LL$, $LR$, $RL$ and $RR$ according to the position of the startpoints and endpoints of each string. An example of a Lorenz braid is shown in Fig. \ref{fig:lorbraid}, where we adopt the convention of drawing the overcrossing ($L$) strings as thicker lines than the undercrossing ($R$) strings. This convention will be used in other braid diagrams.

Each Lorenz braid $\beta$ is a simple braid, i.e., a braid such that all its crossings are positive and every two strings only cross each other at most once, so it has an associated permutation $\pi$. This permutation has only one cycle \emph{iff} it is associated to a knot, and has $k$ cycles if it is associated to a link with $k$ components (knots).

\begin{figure}
  \centering
  \includegraphics[scale=1.0]{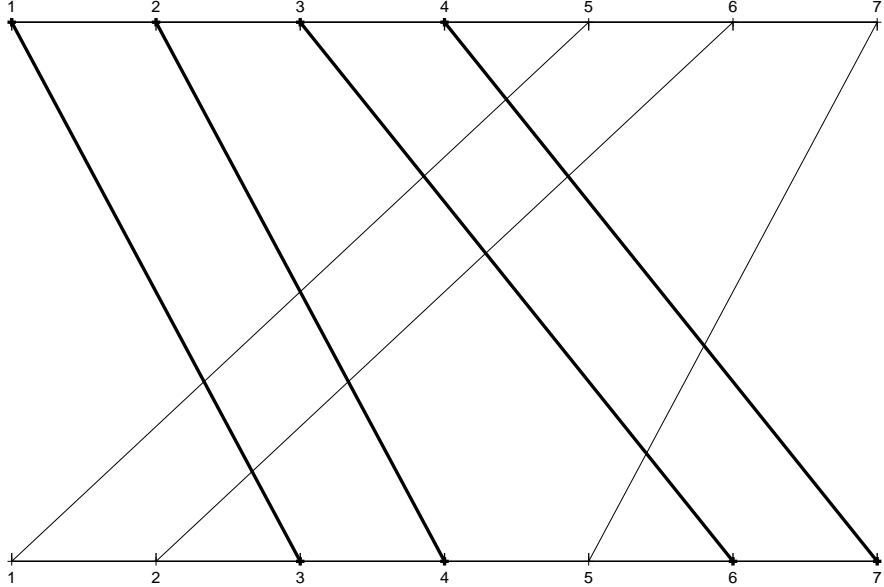}
  \caption{A Lorenz braid}
  \label{fig:lorbraid}
\end{figure}

\subsection*{Symbolic dynamics for the Lorenz map}
\label{sec:symbdynlor}

Let $f^j=f \circ f^{j-1}$ be the $j$-th iterate of the Lorenz map $f$ and $f^0$ be the identity map. We define the \emph{itinerary} of a point $x$ under $f$ as the symbolic sequence $(i_f(x))_j$, $j=0,1,\ldots$ where
$$(i_f(x))_j=\left\{\begin{array}{lll} L & \mathrm{ if } & f^j(x)<0\\
                                       0 & \mathrm{ if } & f^j(x)=0\\
                                       R & \mathrm{ if } & f^j(x)>0. \end{array} \right.$$

 The itinerary of a point in $[-1,1]\setminus \{0\}$ under the Lorenz map can either be an infinite word in the symbols $L,R$ or a finite word in $L,R$ terminated by a single symbol $0$ (because $f$ is undefined at $x=0$).  The \emph{length} $|X|$ of a finite word $X = X_0 \ldots X_{n-1}0$ is $n$, so it can be written as $X = X_0 \ldots X_{|X|-1}0$. A word $X$ is periodic if $X=(X_0 \dots X_{p-1})^{\infty}$ for some $p>1$. If $p$ is the least integer for which this holds, then $p$ is the (least) period of $X$.

The space $\Sigma$ of all finite and infinite words can be ordered in the lexicographic order induced by $L < 0 < R$: given $X, Y \in \Sigma$, let $k$ be the first index such that $X_k \neq Y_k$. Then $X<Y$ if $X_k < Y_k$ and $Y < X$ otherwise.

The \emph{shift map} $s:\Sigma\setminus\{0\} \to \Sigma$ is defined as usual by $s(X_0X_1 \ldots)=X_1 \ldots$ (it just deletes the first symbol). From the definition above, an infinite word $X$ is periodic \emph{iff} there is $p>1$ such that $s^p(X)=X$.  The sequence $W,s(W),\ldots,s^{p-1}(W)$ will also be called the \emph{orbit} of $W$ and a word in the orbit of $W$ will be generally called a shift of $W$.

A (finite or infinite) word $X$ is called \emph{L-maximal}
 if $X_0=L$ and for $k>0$, $X_k = L \Rightarrow s^k(X)\leq X$, and \emph{R-minimal} if $X_0=R$ and for $k>0$, $X_k = R \Rightarrow X \geq s^k(X)$. An infinite periodic word $(X_0 \ldots X_{n-1})^{\infty}$ with least period $n$ is L-maximal (resp. R-minimal) if and only if the finite word $X_0 \ldots X_{n-1}0$ is L-maximal (resp. R-minimal). Therefore, there exists a bijective correspondence between the set of $L$-maximal (resp.$R$-minimal) finite words and the cyclic permutations classes of periodic words.

For a finite word $W$, $n_L(W)$ and $n_R(W)$ will denote respectively the number of $L$ and $R$ symbols in $W$, and $n=n_L+n_R$ the length of $W$. Analogously, if $W$ is periodic with least period $n$ then we define $n_L(W)=n_L(W')$ and $n_R(W)=n_R(W')$, where $W'$ is the $L$-maximal finite word corresponding to $W$. Each periodic word is associated to a Lorenz braid (whose closure is a Lorenz knot), which can be obtained through the following procedure: given a periodic word $W$ with least period  $n$, order the successive shifts $s(W),s^2(W),\ldots,s^n(W)=W$ lexicographically and associate them to startpoints and endpoints in the associated Lorenz braid, with points corresponding to words starting with $L$ lying on the left half and points corresponding to words starting with $R$ on the right half. Each string in the braid connects the startpoint corresponding to $s^k(W)$ to the endpoint corresponding to $s^{k+1}(W)$. Fig. \ref{fig:lorbraid-word} exemplifies this procedure for $W=(LRRLR)^{\infty}$.

\begin{figure}
  \centering
  \includegraphics[scale=1.0]{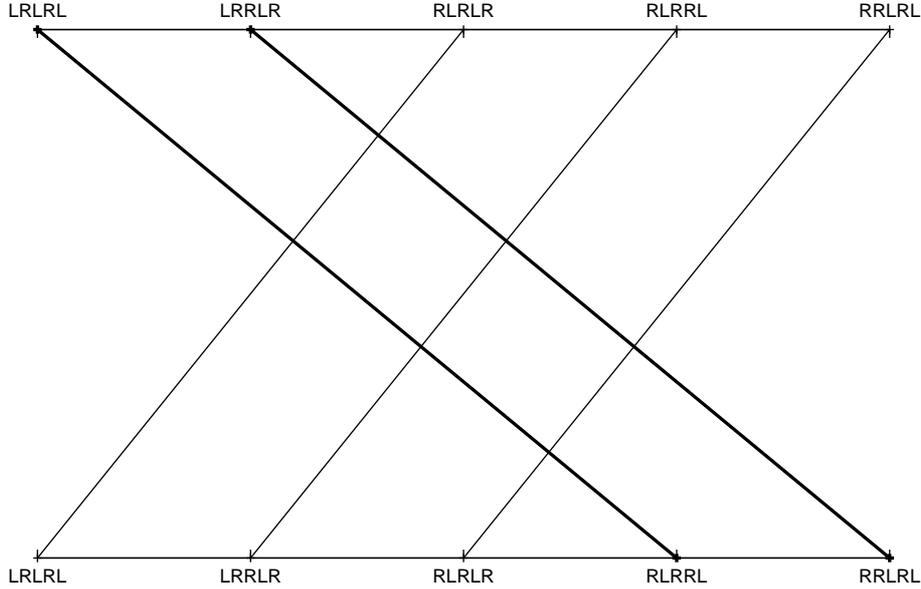}
  \caption{Lorenz braid corresponding to $W=(LRRLR)^{\infty}$}
  \label{fig:lorbraid-word}
\end{figure}

Each periodic orbit of the flow has a unique corresponding orbit in the Lorenz map, which in turn corresponds to the cyclic permutation class of one periodic word in the symbols $L,R$.

Sometimes we will refer to the knot represented by an $L$-maximal or an $R$-minimal word, meaning the knot corresponding to the associated periodic word.

The \emph{crossing number} is the smallest number of crossings in any diagram of a knot $K$. The \emph{braid index} is the smallest number of strings among braids whose closure is $K$.

A \emph{syllable} of a word is a subword of type $L^aR^b$ with maximal length.

The \emph{trip number} $t$ of a periodic word $W$ with least period $n$ is the smallest number of syllables of all its subwords with length $n$. 

The trip number of a Lorenz link is the sum of the trip numbers of its components. Franks and Williams \cite{Franks87}, followed by Waddington \cite{Waddington96}, proved that the braid index of a Lorenz knot is equal to its trip number. This result had previously been conjectured by Birman and Williams \cite{Birman83}.

\section{Syllable permutations of torus knots words}
\label{sec:torsylperm}

It was proved in \cite{Birman83} that all torus knots are Lorenz knots. The torus knot $T(p,q)$ is the closure of a Lorenz braid in $n=p+q$ strings, with $p$ left or $L$ strings that cross over $q$ right or $R$ strings, such that each $L$ string crosses over all the $R$ strings. This Lorenz braid of a torus knot thus has the maximum number of crossings ($pq$) for a Lorenz braid with $p$ $L$ strings and $q$ $R$ strings. Since $T(q,p)=T(p,q)$ we will only consider torus knots $T(p,q)$ with $p<q$. 
The structure of the Lorenz braid of a torus knot $T(p,q)$ ($p<q$) is sketched in Fig. \ref{fig:torbraid}, where only the first and last $L$ strings and some of the $R$ strings are drawn. The remaining $L$ ($R$) strings are parallel to the $L$ ($R$) strings shown.

\begin{figure}
  \centering
  \includegraphics{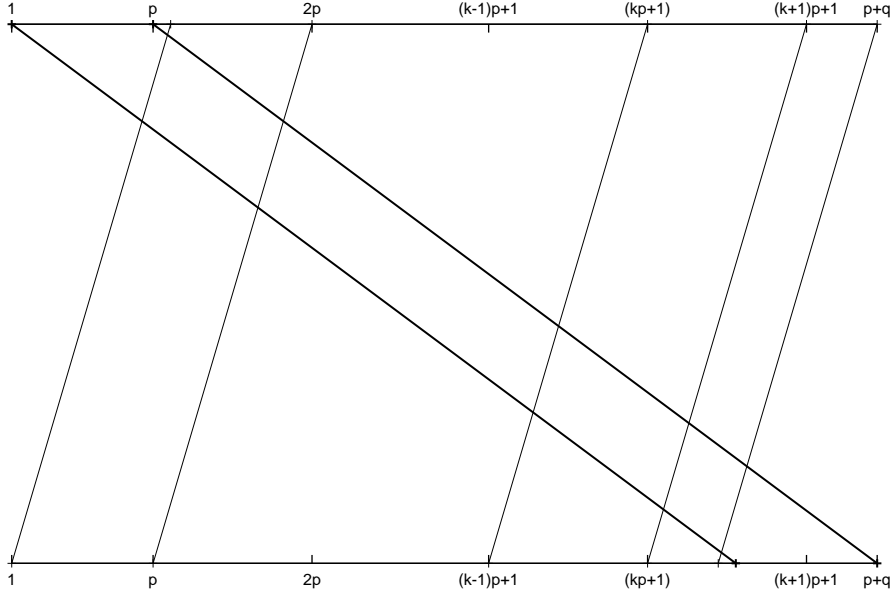}
  \caption{Lorenz braid of $T(p,q) (p<q)$}
  \label{fig:torbraid}
\end{figure}

Lorenz knots corresponding to orbits in the Lorenz template which are represented by evenly distributed words in the alphabet $\{L,R\}$ are torus knots \cite{Birman83}.  On the other hand, given a torus knot $T(p,q)$ there is an evenly distributed word with $n_L=p$, $n_R=q$, that represents it. There is thus a bijection between torus knots and cyclic permutation classes of evenly distributed words. We will call the $L-maximal$ word that represents $T(p,q)$ the \emph{standard word} $W(p,q)$ for $T(p,q)$.

\begin{defin}
  We define $P(p,q)$ as the set of $L$-maximal words resulting from permutations of syllables of the standard word $W(p,q)$ for $T(p,q)$.
\end{defin}

In \cite{PhysicaD} we proved the following results:

\begin{prop}\label{theor:unique}
  For each word $W$ in $P(p,q)$, $4<p<q$, distinct from $W(p,q)$, there is at most one torus knot $T(p,q'),\ q'<q$, with the same braid index and genus as the closure of the braid corresponding to $W$.
\end{prop}

\begin{prop}\label{prop:notorus1}
  For each odd integer $p>4$ and all integer $k>0$, the sets $P(p,q)$, for $q=kp+2$ and $q=(k+1)p-2$ contain no words corresponding to braids whose closure is a torus knot, besides $W(p,q)$.
\end{prop}

\begin{prop}\label{prop:notorus2}
  If $p>4$ is even and not a multiple of $3$, then for any integer $k$ the sets $P(p,kp+3)$ and $P(p,(k+1)p-3)$ contain no words corresponding to braids whose closure is a torus knot. Also, if $p<12$ and $p$ is even, then $P(p,q)$ contains no words other than $W(p,q)$ corresponding to torus knots.
\end{prop}

In \cite{PhysicaD} we derived an algorithm, based in the work of El-Rifai, \cite{Elrifai99}, to obtain Lorenz satellite braids, together with their corresponding aperiodic words and proved that none of the words in the sets $P(p,q)$ can be obtained through this procedure. Thus concluding the following result.

\begin{prop}\label{satellite}
 If Morton's conjecture is true then the Lorenz knots corresponding to syllable permutations of standard torus words, that is, the knots corresponding to words in the sets $P(p,q)$, are not satellites.
 \end{prop}

So we can conclude from Propositions \ref{prop:notorus1} and \ref{prop:notorus2} that, if Morton's conjecture is true, then the words in sets $P(p,kp+2)$, $P(p,(k+1)p-2)$ for $p$ odd and $P(p,kp+3)$, $P(p,(k+1)p-3)$, for $p$ even and $p$ not a multiple of $3$, distinct from the standard $W(p,q)$ torus word, correspond to hyperbolic Lorenz knots.

 Moreover, we have recently performed an extensive computational test \cite{Gomes14}, in which we computed the volumes of all knot complements corresponding to words in the (non-empty) sets $P(p,q)$ with $5\leq p \leq 19$ and $6 \leq q \leq 100$. We found all of them to be hyperbolic, with the expected exception of the torus knots $T(p,q)$ corresponding to the standard words $W(p,q)$.

\section{Farey pairs}

In \cite{DCDS} it was introduced one operation over Lorenz links that is directly related with renormalization of Lorenz maps. Generically, Lorenz maps are one-dimensional maps, $g:[-1,1]\rightarrow [-1,1]$,  with one single discontinuity at $0$, increasing in both continuity intervals and such that $g(\pm 1)=\pm 1$. In particular the first-return map from the original Lorenz template is a Lorenz map in this sense, with the particularity of being surjective in both continuity intervals. If the critical orbits are finite, then Lorenz maps generate sub-Lorenz templates, see \cite{GHS} and \cite{ChSF}. These templates are contained in the Lorenz template, so all knots in them are Lorenz knots. The combinatorics of a Lorenz map $g$, such as the corresponding sub-Lorenz template, are completely determined by its kneading invariant, i.e., by the pair $(X,Y)$, where $X=Li_g(\lim_{x\rightarrow 0^-}g(x))$ and $Y=Ri_g(\lim_{x\rightarrow 0^+}g(x))$ are the critical itineraries.

We say that a pair $(X,Y)\in \Sigma \times \Sigma$ is admissible if it is the kneading invariant of some Lorenz map $g$. 

In \cite{SSR} it was proved the following result.

\begin{prop} A pair $(X,Y)\in \Sigma ^2$ is admissible if and only if the following conditions are verified:
\begin{enumerate}
\item For any $Z\in \lbrace X,Y \rbrace$, if $Z_i=L$ then $\sigma^i(Z)\leq X$.
\item For any $Z\in \lbrace X,Y \rbrace$, if $Z_i=R$ then $\sigma^i(Z)\geq Y$.
\item The previous inequalities are strict if any of the words involved is finite.
\end{enumerate}
\end{prop}

For an admissible pair of finite words $(X,Y)\in \Sigma \times \Sigma$ and a finite word $S\in \Sigma$, we define the $*$-product 
$$(X,Y)*S=\overline{S_0}\ldots \overline{S_{|S|-1}}0,$$
where 
$$\overline{S_j}=\left\lbrace
\begin{array}{ll}
X_0\ldots X_{|X|-1} & \text{ if } S_j =L \\
Y_0\ldots Y_{|Y|-1} & \text{ if } S_j =R
\end{array}
\right.
$$

Words of type $(X,Y)*S$ are the itineraries of points in the renormalization intervals of renormalizable Lorenz maps, and in \cite{DCDS} it was studied the structure of their corresponding Lorenz knots, as a geometric construction depending on the Lorenz link defined by the pair $(X,Y)$ and on the Lorenz knot defined by $S$.

On the other hand \cite{TW}, evenly distributed words in $\Sigma$ are exactly those that can not be written as $(X,Y)*S$ for some admissible pair $(X,Y)$, so torus knots correspond to words that are irreducible relatively to the $*$-product and the hyperbolic and satellite Lorenz knots are generated under the geometric construction derived from the $*$-product. 

The $L$-maximal and $R$-minimal evenly distributed words can be generated recursively in the Symbolic Farey trees constructed below.

First we define the $L$-maximal symbolic Farey tree, $\mathcal{F}^-=\cup_{i=0}^{\infty}\mathcal{F}^-_i$, where  $\mathcal{F}_0^-=\left\lbrace L0 \right\rbrace$ and, for all $n$,
$$\begin{array}{c}
\mathcal{F}^ -_{n+1}= \mathcal{F}^-_n\cup\lbrace LR^{n+1}0\rbrace \cup\\
 \left\lbrace Y_0\ldots Y_{|Y|-1}X_0\ldots X_{|X|-1}0 : X <Y \text{ are consecutive words in }\mathcal{F}^-_n \right\rangle.
\end{array}
$$

Now we say that two $L$-maximal words $X<Y$ are Farey  neighbours if there is some $n$ such that they are consecutive words in $\mathcal{F}^-_n$.

So we have $\mathcal{F}^-$:
$$\begin{array}{ccccccccc}
 L0 &     &    &      &     &      &    & &     \\
    &     &    &      & LR0 &      &    & &      \\
    &     &LRL0&      &     &      &LRR0& &      \\
    &LRLL0&    &LRLRL0&     &LRRLR0&    &LRRR0 &  \\
    & \:\:\:  \vdots& &\:\:\:  \vdots & &\:\:\:  \vdots & &  \:\:\:  \vdots &
\end{array}
$$

We define analogously the $R$-minimal symbolic Farey tree, $\mathcal{F}^+=\cup_{i=0}^{\infty}\mathcal{F}^+_i$, where  $\mathcal{F}_0^+=\left\lbrace R0 \right\rbrace$ and, for all $n$,
$$\begin{array}{c}
\mathcal{F}^ +_{n+1}= \mathcal{F}^+_n\cup\lbrace RL^{n+1}0\rbrace \cup\\
 \left\lbrace X_0\ldots X_{|X|-1}Y_0\ldots Y_{|Y|-1}0 : X <Y \text{ are consecutive words in }\mathcal{F}^+_n \right\rangle.
\end{array}
$$

Finally we have $\mathcal{F}^+$:
$$\begin{array}{ccccccccc}
  &     &    &      &     &      &    & &R0     \\
    &     &    &      & RL0 &      &    & &      \\
    &     &RLL0&      &     &      &RLR0& &      \\
    &RLLL0&    &RLLRL0&     &RLRRL0&    &RLRR0  & \\
    &\:\:\:  \vdots & &\:\:\:  \vdots & & \:\:\:  \vdots& & \:\:\:  \vdots & 
\end{array}
$$

For a word $X\in\mathcal{F}^-$, we may define its $R$-minimal version $m(X)=\min\lbrace X_j\ldots X_{|X|-1}X_0\ldots X_{j-1}0 : X_j=R\rbrace $. It is immediate to observe that each word  $Y\in\mathcal{F}^+\setminus\mathcal{F}^+_0$ is obtained as $m(X)$ where $X\in\mathcal{F}^-\setminus\mathcal{F}^-_0$ is in the same position of $\mathcal{F}^-\setminus\mathcal{F}^-_0$.

\begin{defin}
A Farey pair is a pair $(X,Y)$ where $Y=m(S)$, $X,S\in \mathcal{F}^-$, and $X$ and $S$ are Farey neighbours with $S<X$.
\end{defin}

From the point of view of the genealogy of Lorenz words, see \cite{SSR}, words of type $(X,Y)*S$ where $(X,Y)$  are Farey pairs, are the "first" reducible words. In the rest of this paper, we will prove that, if Morton's conjecture is true, then  those words never correspond to satellite Lorenz knots and, for some specific families of Farey pairs $(X,Y)$ they all correspond to hyperbolic knots. However, the computational tests performed in \cite{Gomes14}, lead us to conjecture that all words of the referred type correspond to hyperbolic knots.

\begin{lem}\label{farneighadmiss}
  If $(X,Y)$ is a Farey pair, then $(X,Y)$ is admissible.
\end{lem}

\begin{proof}
  The proof follows immediately from the construction of the symbolic Farey trees.
\end{proof}

\begin{thm}\label{fareyneighstarprod}
Let $(X,Y)$ be a Farey pair such that $t(X) > 1$ and $t(Y) > 1$. Let $p_1 = \min \{n_L(X),n_R(X)\}$,  $q_1 = \max \{n_L(X),n_R(X)\}$, $p_2 = \min \{n_L(Y),n_R(Y)\}$, $q_2 = \max \{n_L(Y),n_R(Y)\}$, $q_1=kp_1+r_1\ (0<r_1<p_1)$, $q_2=kp_2+r_2\ (0<r_2<p_2)$. The Lorenz knots associated to $X$ and $Y$ are respectively the torus knots $T(p_1,q_1)$ and $T(p_2,q_2)$. Then, for any finite aperiodic word $S\in\Sigma$, $Z=(X,Y)*S$ is a nontrivial syllable permutation of the standard evenly distributed word associated to the torus knot $T(p,q)$, where $q=kp + r (1<r<p-1)$, where $p=n_L(S)p_1 + n_R(S)p_2 =$, $q=n_L(S)q_1 + n_R(S)q_2$, $r=n_L(S)r_1 + n_R(S)r_2$.
\end{thm}

\begin{proof}
  $Z=(\overline{S}_0\overline{S}_1 \dots \overline{S}_{|S|-1})^{\infty}$ where $$\overline{S}_i= \begin{cases} X_0 X_1 \dots X_{p_1+q_1-1} &\text{if } S_i=L\\ Y_0 Y_1 \dots Y_{p_2+q_2-1} &\text{if } S_i=R \end{cases}.$$

First assume $n_R(X)<n_L(X)$, $n_R(Y)<n_L(Y)$. $X$ and $Y$ must have the form $X = LRL^k \dots RL^k0$, $Y=RL^{k+1} \dots RL^k0$. In $Z$, pairs of consecutive subwords are therefore of one of the following types:
\begin{itemize}
\item $XX=LRL^k \dots RL^k LRL^k \dots RL^k = LRL^k \dots RL^{k+1}RL^k \dots RL^k$
\item $XY=LRL^k \dots RL^k RL^{k+1} \dots RL^k$
\item $YX=RL^{k+1} \dots RL^k LRL^k \dots RL^k=RL^{k+1} \dots RL^{k+1}RL^k \dots RL^k$
\item $YY=RL^{k+1} \dots RL^k RL^{k+1} \dots RL^k$
\end{itemize}

Therefore, any shift of $Z$ starting with an $R$ has syllables of only two types: $RL^k$ and $RL^{k+1}$ and must therefore be a syllable permutation of the torus knot $T(p,q)$, where $p=n_R(Z)=n_L(S) p_1+n_R(S) p_2$ and $q=n_L(Z)=n_L(S) q_1+n_R(S) q_2$. Since $Z$ is reducible, $Z$ is not in the Farey tree and the syllable permutation must be nontrivial. Finally, $q = n_L(S) (kp_1+r_1) + n_R(S) (kp_2+r_2) = kp+r$, where $r=n_L(S) r_1 + n_R(S)r_2$. Since $0<r_1<p_1$ and $0<r_2<p_2$ and $r_1,r_2$ are integers, $1<r<p-1$.

Now assume $n_L(X)<n_R(X)$, $n_L(Y)<n_R(Y)$. Since $X=LR^{k+1} \dots LR^k$ and $m(Y)= RLR^k \dots LR^k$, pairs of consecutive words in $Z$ will now be of one of the following types:

\begin{itemize}
\item $XX=LR^{k+1} \dots LR^k LR^{k+1} \dots LR^k$
\item $XY=LR^{k+1} \dots LR^k RLR^k \dots LR^k = LR^{k+1} \dots LR^{k+1} LR^k \dots LR^k$
\item $YX=RLR^k \dots LR^k LR^{k+1} \dots LR^k$
\item $YY=RLR^k \dots LR^k RLR^k \dots LR^k = RLR^k \dots LR^{k+1} LR^k \dots LR^k$
\end{itemize}

In this case, any shift of $Z$ starting with an $L$ has syllables of only two types: $LR^k$ and $LR^{k+1}$ and is therefore a syllable permutation of the torus knot $T(p,q)$, where $p=n_L(Z)=n_L(S) p_1+n_R(S) p_2$ and $q=n_R(Z)=n_L(S) q_1+n_R(S) q_2$. Again, $Z$ is reducible and therefore the permutation is nontrivial. Using the same argument as in the previous case we conclude that $1<r<p-1$.

If $n_L(X)<n_R(X)$ and $n_L(Y)<n_R(Y)$ then the proof follows analogously. Now, since for every Farey pair $(X,Y)$ satisfying $t(X)>1,\ t(Y)>1$ \label{fareyneighstarprod} we have either $Y < X < LR0 $ or $ LR0 < Y < X$ and therefore  either $n_L(X)<n_R(X)$, $n_L(Y)<n_R(Y)$ or $n_R(X)<n_L(X)$, $n_R(Y)<n_L(Y)$, the result is proved for any Farey pair.
\end{proof}

So, since all torus knots have a corresponding word in the symbolic Farey trees, from Proposition \ref{satellite}, we conclude that, for Farey pairs $(X,Y)$, if $(X,Y)*S$ generates "new" knots, then they are hyperbolic.

In spite of the extensive computational tests indicate that all knots corresponding to pairs 
$(X,Y)*S$ for any Farey pair $(X,Y)$ are indeed hyperbolic, the results from \cite{PhysicaD} only allow us to demonstrate this for the following families.

\begin{corollary}\label{hyperbolicstarprod}

The Lorenz knots associated to the following families, all of which are of the type defined in Theorem \ref{fareyneighstarprod}, are hyperbolic:

\begin{enumerate}
\item $\left( L(RL^k)^{n+1}0, RL^{k+1}(RL^k)^{n-1}0 \right) \ast (LR)^{\infty}$, $k>0,\ n>1$ 

\item $\left( LRL^k (RL^{k+1})^{n-2} RL^k0, (RL^{k+1})^n RL^k0 \right) \ast (LR)^{\infty}$, $k>0,\ n>1$ 

\item $\left( L(RL^k)^n0, RL^{k+1} (RL^k)^{n-2} RL^{k+1} (RL^k)^{n-1}0 \right) \ast (LR)^{\infty}$, $k>0,\ n >1 \text{ odd}$ 

\item $\left( L(RL^k)^n RL^{k+1} (RL^k)^n0, RL^{k+1} (RL^k)^{n-1}0 \right) \ast (LR)^{\infty}$, $k>0,\ n>1 \text{ odd}$ 

\item $\left( L(RL^k)^{n+1}0, RL^{k+1}(RL^k)^{n-1}0 \right) \ast (LRL)^{\infty}$, $k>0,\ n\text{ even },\ n>1$ 

\item $\left( L(RL^k)^{n+1}0, RL^{k+1}(RL^k)^{n-1}0 \right) \ast (LRR)^{\infty}$, $k>0,\ n\text{ odd },\ n>1$ 

\item $\left( LRL^k (RL^{k+1})^{n-2} RL^k0, (RL^{k+1})^n RL^k (RL^{k+1})^{n-1} RL^k0 \right) \ast (LR)^{\infty}$,\newline $k>0,\ n>1, \text{ odd}$ 

\item $\left( LRL^k (RL^{k+1})^{n-2} RL^k (RL^{k+1})^{n-2} RL^k0, (RL^{k+1})^{n-1} RL^k0 \right) \ast (LR)^{\infty}$,\newline $k>0,\ n>1 \text{ odd}$ 

\item $\left( LRL^k (RL^{k+1})^{n-2} RL^k0, (RL^{k+1})^n RL^k0 \right) \ast (LRL)^{\infty}$, $k>0,\ n >1 \text{ odd }$ 

\item $\left( LRL^k (RL^{k+1})^{n-2} RL^k0, (RL^{k+1})^n RL^k0 \right) \ast (LRR)^{\infty}$, $k>0,\ n > 1 \text{ even }$ 
\end{enumerate}
  
\end{corollary}

\begin{proof}
  From Lemma \ref{farneighadmiss}, to prove that each of the pairs $(X,Y)$ above is admissible it is sufficient to show that they are Farey pairs. We then use Theorem \ref{fareyneighstarprod} to prove that each product $(X,Y) \ast S$ is a syllable  permutation of the standard $T(p,q)$ word, with $p,q$ satisfying either $q=kp+2$ or $q=(k+1)p-2$ with $p>4$ odd, or $q=kp+3$ or $q=(k+1)p-3$ with $p$ even and not a multiple of $3$. 

  \begin{enumerate}
  \item For this family, $(X,Y) = \left( L(RL^k)^{n+1}0, RL^{k+1}(RL^k)^{n-1}0 \right)$, so $Y = m(L(RL^k)^n0)$ and $L(RL^k)^n0<X$. We start by remarking that $L0,LR0$ are Farey neighbours in level 1 of the maximal Farey tree. Therefore  $L0,LRL0$  are also Farey neighbours. Since for $k>1$ $LRL^k0=LRL^{k-1}L0$ is the concatenation of $LRL^{k-1}0$ and $L0$, $LRL^k0$ and $LRL^{k-1}0$ are Farey neighbours, $k>0$. Finally, $L(RL^k)^{n+1}0=\left(LRL^{k-1}\right)^nLRL^k0$ therefore $L(RL^k)^{n+1}0, L(RL^k)^n0$ are Farey neighbours and $\big( L(RL^k)^{n+1}0,\allowbreak RL^{k+1}(RL^k)^{n-1}0 \big)$ is a Farey pair.

Since $p_1=n_R(X)=n+1>2$, $p_2 = n_R(Y)=n>1$ and $n_L(S)=n_R(S)=1$, the trip number $p=p_1+p_2=2n+1$ is odd. From $q_1=n_L(X)=(n+1)k+1=kp_1+1)$ and $q_2=k+1+(n-1)k=kn+1=kp_2+1$, so $r=r_1+r_2=2$ and $(X,Y) \ast S$ is a syllable permutation of the standard word corresponding to $T(p,q)=T(p,kp+2)$, $p>4$ odd and therefore corresponds to a hyperbolic knot.
 
  \item The pair $(X,Y) = \big( LRL^k (RL^{k+1})^{n-2} RL^k0, (RL^{k+1})^n RL^k0 \big)$, so we have $Y = m(LRL^k (RL^{k+1})^{n-1} RL^k0)$ and $RL^{k+1})^{n-1} RL^k0<X$. From the previous case, $LRL^k0$ and $LRL^{k-1}0$ are Farey neighbours.

Since $LRL^kRL^k0=LRL^{k-1} LRL^k0$, $LRL^k RL^k0$ and $LRL^k0$ are also Farey neighbours and therefore $LRL^k RL^{k+1} RL^k0 = LRL^k RL^k LRL^k 0$ and $LRL^k0$ are also Farey neighbours.

For $n>2$, $LRL^k (RL^{k+1})^{n-2} RL^k0 = LRL^k LRL^k (RL^{k+1})^{n-1} RL^k0$, so $LRL^k (RL^{k+1})^{n-2} RL^k0$ and $LRL^k (RL^{k+1})^{n-1} RL^k0$ are Farey neighbours and $(X,Y)$ is therefore a Farey pair.
 
For this family $p_1=n$, $p_2=n+1$, so $p=2n+1>4$ is odd, while $r_1=p_1-1$, $r_2=p_2-1$ and therefore $r=p-2$. $(X,Y) \ast S$ is thus a syllable permutation of the standard word for $T(p,k(p+1)-2)$.

  \item In this case $(X,Y) = \left( L(RL^k)^n0, RL^{k+1} (RL^k)^{n-2} RL^{k+1} (RL^k)^{n-1}0 \right)$, so  $Y =\allowbreak m( L(RL^k)^n L(RL^k)^{n-1}0)$ and $ L(RL^k)^n L(RL^k)^{n-1}0<X $. $L(RL^k)^n0$ and $L(RL^k)^{n-1}0$ are Farey neighbours since $L(RL^k)^n0=(LRL^{k-1})^{n-1}LRL^k0$ and $LRL^{k-1},LRL^k0$ are Farey neighbours as seen in 1. Therefore, $X$ and  $L(RL^k)^n L(RL^k)^{n-1}0$ are also Farey neighbours and $(X,Y)$ is Farey.

The trip number $p=n + 1 + n-2 +1 +n-1 =3n-1$ is even and not divisible by 3. Since $r_1=1,\ +r_2=2$, $r=3$ and $(X,Y) \ast S$ is a syllable permutation of the Farey word for $T(p,kp+3)$, $p even$, $p>4$.

  \item In family 4, $(X,Y) = \left( L(RL^k)^n RL^{k+1} (RL^k)^n0, RL^{k+1} (RL^k)^{n-1}0 \right)$ and $Y=m(L(RL^k)^n0)$, so $X =L(RL^k)^{n+1}  L(RL^k)^n0$ and from the previous proof  $X$ and $L(RL^k)^n0$ are Farey neighbours and $(X,Y)$ is a Farey pair.

The trip number $p=2n+1 + n = 3n+1$ is again even and not divisible by 3. Since $r_1=2,\ +r_2=1$, $r=3$ and $(X,Y) \ast S$ is a syllable permutation of the Farey word for $T(p,kp+3)$, $p even$, $p>4$.

\item $(X,Y)$ is identical to the pair in family 1 and therefore a Farey pair.

For the trip number we have $p=n_L(S)p_1 + n_R(S)p_2=2(n+1) + n = 3n+2$ even and not divisible by $3$; $r=2r_1+r_2=3$. $(X,Y) \ast S$ is again a syllable permutation of the Farey word for $T(p,kp+3)$, $p$ even, $p>4$.

\item $(X,Y)$ is Farey (it is again identical to the pair in family 1).

The trip number is given by $p=n_L(S)p_1 + n_R(S)p_2=n+1 + 2n =3n+1$ (even since $n$ is odd) while $r=r_1+2r_2=3$ and therefore $(X,Y) \ast S$ is once more a syllable permutation of the Farey word for $T(p,kp+3)$, $p even$, $p>4$.

  \item Here $(X,Y) = \left( LRL^k (RL^{k+1})^{n-2} RL^k0, (RL^{k+1})^n RL^k (RL^{k+1})^{n-1} RL^k0 \right)$. We have $Y = m(L(RL^k) (RL^{k+1})^{n-1} RL^k (RL^{k+1})^{n-1} RL^k0) = m(XU)$ with $U= LRL^k (RL^{k+1})^{n-1} RL^k0$. Since $U= LRL^k (RL^{k+1})^{n-2} RL\ LRL^k0\allowbreak =LRL^{k-1}(LRL^k)^{n-1}0$ and $LRL^{k-1}0,LRL^k0$ are Farey neighbours as seen above, $X,U$ and therefore $X,XU$ are also Farey neighbours. Therefore, $(X,Y)$ is a Farey pair.

For this family $p=n + 2n+1=3n+1$ is even since $n$ is odd and not divisible by 3, while $r=(p_1-1)+(p_2-2)$=$p-3$. Therefore, $(X,Y)*S$ is a syllable permutation of the Farey word corresponding to $T(p,k(p+1)-3)$.

\item $X=LRL^k (RL^{k+1})^{n-2} RL^k (RL^{k+1})^{n-2} RL^k0$, $Y=m(LRL^k (RL^{k+1})^{n-2} RL^k0)$ so  $X=VLRL^k (RL^{k+1})^{n-2} RL^k0$ with $V=LRL^k(RL^{k+1})^{n-3}RL^k$. Since $V=LRL^{k-1}(LRL^k)^{n-2}$, we can write $LRL^k (RL^{k+1})^{n-2} RL^k0=V\ LRL^k$,  so $X,LRL^k (RL^{k+1})^{n-2} RL^k0$ are Farey neighbours and $(X,Y)$ is a Farey pair.

In this case $p=2n-1 + n= 3n -1$ is even since $n$ is odd and not a multiple of $3$. We have $r=(p_1-2)+(p_2-1)$=$p-3$. $(X,Y)*S$ is thus a syllable permutation of the Farey word corresponding to $T(p,k(p+1)-3)$.

\item $(X,Y)$ is the same as in family 2, so $p_1=n$, $p_2=n+1$ and $p=2p_1+p_2=3n+1$ is even for $n$ odd, and undivisible by $3$. $r=2r_1+r_2=2(p_1-1)+p_2-1=p-3$, therefore $(X,Y)*S$ is thus a syllable permutation of the Farey word corresponding to $T(p,k(p+1)-3)$.

\item $(X,Y)$ is again the same as in family 2, so $p_1=n$, $p_2=n+1$ and $p=p_1+2p_2= 3n+2$ is even since $n$ is even and not a multiple of $3$. $r=r_1+2r_2=p_1-1 + 2(p_2-1)=p-3$, so $(X,Y)*S$ is a syllable permutation of the Farey word corresponding to $T(p,k(p+1)-3)$.
  \end{enumerate}

\end{proof}

\begin{remark}\label{exchangeLR}
  To these families we can add those obtained by exchanging $L$ and $R$ in all words. More precisely, let $\hat{X}$, $\hat{Y}$ and $\hat{S}$ be the words obtained by exchanging $L$ and $R$ in $X$, $Y$ and $S$, respectively. Then due to the simmetry in the Farey tree, if $(X,Y)$ is a Farey pair and $(X,Y) * S$ corresponds to a  hyperbolic knot, $(\hat{Y},\hat{X})$ is a Farey pair and $(\hat{Y},\hat{X}) * \hat{S}$ is an R-minimal word corresponding to the same hyperbolic Lorenz knot. 
\end{remark}

\begin{remark}
  From the list of hyperbolic Lorenz knots presented by Birmann and Kofman in \cite{Birman09}, eight can be represented by syllable permutations of torus knot words. All the words representing these eight knots belong to one of the families defined in Corollary \ref{hyperbolicstarprod}.
\end{remark}

\begin{remark}
  The families defined in Corollary \ref{hyperbolicstarprod} and Remark \ref{exchangeLR} contain all the words of type $(X,Y)*S$, up to shifting, that belong to one of the sets $P(p,kp+2)$, $P(p,k(p+1)-2)$, $P(p,kp+3)$ or $P(p,k(p+1)-3)$ of syllable permutations defined in propositions \ref{prop:notorus1} and \ref{prop:notorus2}.
\end{remark}

\end{document}